\documentclass[10pt, oneside, english]{article}

\usepackage[english]{babel}
\usepackage[utf8x]{inputenc}
\usepackage[T1]{fontenc}
\usepackage[letterpaper,top=1in,bottom=1in,left=1in,right=1in,marginparwidth=1.75cm]{geometry}
\usepackage{dsfont}
\usepackage{amsmath, amssymb, amsthm, verbatim,enumerate,bbm}
\usepackage{indentfirst, bbm}
\usepackage{appendix}
\usepackage{enumerate}
\usepackage{verbatim}
\usepackage[colorinlistoftodos]{todonotes}
\usepackage[colorlinks=true, allcolors=blue]{hyperref}
\usepackage[nameinlink,capitalise,noabbrev]{cleveref}
\crefname{appsec}{Appendix}{Appendices}
\creflabelformat{enumi}{#2textup{#1}#3}

\allowdisplaybreaks
\makeatletter
\newtheorem{theorem}{Theorem}[section]

\newtheorem*{namedtheorem}{\theoremname}
\newcommand{\theoremname}{testing}

\newtheorem{lemma}[theorem]{Lemma}

\newtheorem{proposition}[theorem]{Proposition}

\newtheorem{corollary}[theorem]{Corollary}

\newtheorem*{question*}{Question}

\theoremstyle{definition}
\newtheorem{definition}[theorem]{Definition}

\newtheorem{remark}[theorem]{Remark}

\theoremstyle{plain}

\makeatother

\title{A note on the universality of ESDs of inhomogeneous random matrices}
\author{
Vishesh Jain\thanks{Email: {\tt vishesh.vj@gmail.com}. 
}
\and 
Sandeep Silwal\thanks{Massachusetts Institute of Technology. Department of Computer Science. Email: {\tt silwal@mit.edu}. 
}
}
\date{}

\DeclareMathOperator{\LCD}{LCD}
\DeclareMathOperator{\CRLCD}{CRLCD}
\DeclareMathOperator{\Ber}{Ber}
\DeclareMathOperator{\dist}{dist}
\DeclareMathOperator{\Sparse}{Sparse}
\DeclareMathOperator{\Comp}{Comp}
\DeclareMathOperator{\Incomp}{Incomp}
\DeclareMathOperator{\HS}{HS}

\begin{document}
\maketitle
\global\long\def\R{\mathbb{R}}

\global\long\def\S{\mathbb{S}}

\global\long\def\Z{\mathbb{Z}}

\global\long\def\C{\mathbb{C}}

\global\long\def\Q{\mathbb{Q}}

\global\long\def\N{\mathbb{N}}

\global\long\def\P{\mathbb{P}}

\global\long\def\F{\mathbb{F}}

\global\long\def\U{\mathcal{U}}

\global\long\def\V{\mathcal{V}}

\global\long\def\E{\mathbb{E}}

\global\long\def\Ev{\mathscr{Rk}}

\global\long\def\Dg{\mathscr{D}}

\global\long\def\Ndg{\mathscr{ND}}

\global\long\def\Rv{\mathcal{R}}

\global\long\def\Gv{\mathscr{Null}}

\global\long\def\Hv{\mathscr{Orth}}

\global\long\def\Supp{{\bf Supp}}

\global\long\def\Sv{\mathscr{Spt}}

\global\long\def\ring{\mathfrak{R}}

\global\long\def\1{\mathbbm{1}}

\global\long\def\Bad{{\boldsymbol{B}}}

\global\long\def\supp{{\bf supp}}

\global\long\def\A{\mathcal{A}}

\global\long\def\L{\mathcal{L}}


\begin{abstract}
In this short note, we extend the celebrated results of Tao and Vu, and Krishnapur on the universality of empirical spectral distributions to a wide class of \emph{inhomogeneous} complex random matrices, by showing that a technical and hard-to-verify Fourier domination assumption may be replaced simply by a natural uniform anti-concentration assumption. 

Along the way, we show that inhomogeneous complex random matrices, whose expected squared Hilbert-Schmidt norm is quadratic in the dimension, and whose entries (after symmetrization) are uniformly anti-concentrated at $0$ and infinity, typically have smallest singular value $\Omega(n^{-1/2})$. The rate $n^{-1/2}$ is sharp, and closes a gap in the literature.

Our proofs closely follow recent works of Livshyts, and Livshyts, Tikhomirov, and Vershynin on inhomogeneous \emph{real} random matrices. The new ingredient is a couple of anti-concentration inequalities for sums of independent, but not necessarily identically distributed, complex random variables, which may also be useful in other contexts. 
\end{abstract}

\section{Introduction}

\subsection{The least singular value of inhomogeneous, heavy-tailed, complex random matrices} 
The (ordered) \emph{singular values} of an $n\times n$ complex matrix $A_{n}$, denoted by $s_k(A_n)$ for $k\in [n]$, are defined to be the eigenvalues of $\sqrt{A_{n}^{\dagger}A_n}$ arranged in non-decreasing order. Recall that the extreme singular values $s_1(A_n)$ and $s_n(A_n)$ admit the following variational characterization:
\begin{align}\label{eq:characterization}
    s_1(A_n):= \sup_{x\in \S^{n-1}_{\C}}\|A_{n}x\|_{2},    \quad 
  s_n(A_n):= \inf_{x\in \S^{n-1}_{\C}}\|A_{n}x\|_{2},  
\end{align}
where $\|\cdot \|_{2}$ denotes the standard Euclidean norm in $\C^{n}$, and $\S_{\C}^{n-1}$ denotes standard unit sphere in $\C^{n}$. 
In this short note, we will primarily be concerned with the non-asymptotic study of the smallest singular value $s_n(A_n)$ (for quite general random matrices $A_n$) -- a subject which has its origins in numerical linear algebra, and which has attracted much attention in recent years (see, for instance, the references in \cite{livshyts2019smallest}). 

When the entries of $A_n$ are i.i.d. complex Gaussians, Edelman \cite{edelman1988eigenvalues} showed that for any $\epsilon > 0$,
$$\Pr\left(s_n(A_n) \leq \epsilon n^{-1/2}\right) \leq \epsilon^{2};$$
in particular, this shows that for any $\delta > 0$, with probability at least $1-\delta$, $s_n(A_n) = \Omega_{\delta}(n^{-1/2})$. In other words, the smallest singular value of a `typical realization' of an i.i.d. complex Gaussian matrix is at least order $n^{-1/2}$ (which is known to be optimal).

As our first main result, we establish the optimal order of $s_n(A_n)$ for a typical realization of $A_n$ for very general ensembles of random matrices -- this is a complex analogue of a recent theorem of Livshyts \cite{livshyts2018smallest} (see the discussion below). 
\begin{theorem}
\label{thm:lsv-bound}
Let $A_n$ be an $n\times n$ complex random matrix whose entries $A_{i,j}$ are independent and satisfy the following two conditions: $\sum_{i,j}\E|A_{i,j}|^{2}   \leq Kn^{2}$ for some $K > 0$ and $\Pr\left(b^{-1} \geq |\widetilde{A_{i,j}}| \geq b \right) \geq b$ for some $b\in (0,1)$ (here, $\widetilde{A_{i,j}}$ denotes the difference of two independent copies of $A_{i,j}$). Then, for all $\epsilon \in [0,1)$,
$$\Pr\left(s_n(A_n)\leq \frac{\epsilon}{\sqrt{n}}\right) \leq C\left( \epsilon + \exp(-c\epsilon^{2}n)
\right),$$
where $C,c$ depend only on $K$ and $b$.
\end{theorem}

In particular, \cref{thm:lsv-bound} implies that for any fixed $\delta>0$, with probability at least $1-\delta$, $s_n(A_n) \geq \Omega_{\delta}(n^{-1/2})$. The rate $n^{-1/2}$ is optimal, and to the best of our knowledge, all previous works considering general heavy-tailed complex random matrices miss this sharp rate. For instance, it was shown by Tao and Vu \cite{tao2010smooth} that if the entries of  $A_{i,j}$ are dominated (in a technical Fourier sense) by a complex random variable with $\kappa$-controlled second moment (see \cite{tao2010smooth} for definitions), then for any $C, \alpha > 0$,
\begin{equation}
\label{eqn:TV-bound}
\Pr(s_n(A_n) \leq n^{-C}\cdot n^{-1/2}) \lesssim_{C,\alpha} n^{-C+\alpha + o_n(1)} + \Pr(\|A_{n}\| \geq n^{1/2}).
\end{equation}
The technical Fourier-domination condition needed for the above result already implies that $\Pr( b^{-1} \geq |\widetilde{A_{i,j}}| \geq b) \geq b$ for some $b\in (0,1)$. Moreover, in order for the term $\Pr\left(\|A_{n}\| \geq n^{1/2}\right)$ to be bounded away from $1$, one needs to further assume more restrictive assumptions on the second moments than in \cref{thm:lsv-bound}, along with the assumption that $\sum_{i,j = 1}^{n}\E|A_{i,j}|^{4} \leq Kn^{2}$ for some $K > 0$.

For the case when the $A_{i,j}$ are i.i.d., $\Pr\left( |\widetilde{A_{i,j}}|\geq b\right) \geq b$ for some $b\in (0,1)$, and $\E|A_{i,j}|^{2} \leq K$ for some $K > 0$, the first author showed \cite{jain2019strong} that for any $\epsilon, \alpha > 0$,
\begin{equation}
    \Pr\left(s_n(A_n) \leq \epsilon n^{-1/2-\alpha}\right) \lesssim _{\alpha} \epsilon + \exp(-cn^{1/50}),
\end{equation}
which again misses the correct rate.\\ 

This somewhat dire situation in the general complex case should be contrasted with the real case, where much more is known. The early breakthrough of Rudelson \cite{rudelson2008invertibility} established that for an $n\times n$ matrix $A_{n}$ whose entries are i.i.d. copies of a real centered sub-Gaussian random variable, for any $\delta >0$, $s_{n}(A_{n}) = \Omega_{\delta}(n^{-1/2})$ with probability at least $1-\delta$. A subsequent breakthrough of Rudelson and Vershynin \cite{rudelson2008littlewood} refined this to the near-optimal tail bound
\begin{equation}
\label{eqn:RV}
\Pr(s_n(A_n)\leq \epsilon\cdot n^{-1/2})\lesssim \epsilon + \exp(-cn).
\end{equation}
Extensions of the above tail bound to heavy-tailed and inhomogeneous matrices has attracted much attention in recent years. Rebrova and Tikhomirov \cite{rebrova2018coverings} extended Rudelson and Vershynin's result to the case when the sub-Gaussian assumption is replaced by the finiteness of the second moment (the entries are still assumed to be identically distributed and centered). Livshyts \cite{livshyts2018smallest} showed that if the entries $A_{i,j}$ are independent \emph{real} variables, $\Pr\left(|\widetilde{A_{i,j}}| \geq b\right) \geq b$ for some $b\in (0,1)$, and $\sum_{i,j=1}^{n} \E |A_{i,j}|^{2} \leq Kn^{2}$ for some $K > 0$, then 
\begin{equation}
    \Pr\left(s_{n}(A_n) \leq \epsilon \cdot n^{-1/2}\right) \lesssim \epsilon + n^{-1/2}.
\end{equation}
Finally, Livshyts, Rudelson, and Tikhomirov \cite{livshyts2019smallest} obtained the near-optimal tail estimate \cref{eqn:RV} under these assumptions. \\ 

Perhaps unsurprisingly, our proof makes use of tools introduced in \cite{livshyts2018smallest, livshyts2019smallest}. The key new ingredient is an anti-concentration inequality for sums of independent complex random variables, which we will discuss in \cref{sec:anti-concentration}. \\

\subsection{Universality of ESDs of dense, inhomogeneous random matrices} 
The \emph{empirical spectral distribution} (ESD) $\mu_{n}$ of an $n\times n$ complex matrix $A_n$ is defined on $\R^{2}$ by the expression
$$\mu_{n}(s,t) := \frac{1}{n}\cdot \left|\{k \in [n] \mid \Re(\lambda_k) \leq s; \Im(\lambda_k) \leq t\}\right|,$$
where $\lambda_1,\dots,\lambda_n$ denote the eigenvalues of $A_n$. A major highlight of random matrix theory is the celebrated \emph{circular law} of Tao and Vu \cite{tao2010random}, which asserts that for $\emph{any}$ fixed complex random variable $x$ of mean $0$ and variance $1$, the ESD of $A_{n}/\sqrt{n}$ -- where $A_{n}$ is an $n\times n$ random matrix each of whose entries is an independent copy of $x$ -- converges uniformly to the distribution of the uniform measure on the unit disc in the complex plane,
$$\mu_{\infty}(s,t):= \frac{1}{\pi}\text{area}\{x\in \C \mid |x| \leq 1, \Re(x) \leq s, \Im(x) \leq t\}$$
as $n$ tends to infinity. More generally, Tao and Vu showed that for any fixed complex random variables $x$ and $y$ of mean $0$ and variance $1$, and for any sequence of deterministic matrices $M_{n}$ satisfying $\|M_{n}\|_{\HS}^{2} = O(n^{2})$, the ESDs of $(M_{n} + X_{n})/\sqrt{n}$ and $(M_{n} + Y_{n})/\sqrt{n}$ convergence in probability to $0$, where $X_{n}$ is an $n\times n$ random matrix whose entries are i.i.d. copies of $x$, and $Y_{n}$ is an $n\times n$ random matrix whose entries are i.i.d. copies of $y$. These results were extended by Krishnapur \cite{tao2010random} to independent, but not necessarily identically distributed matrices, satisfying certain restrictions on the distributions of the entries.  

Here, by using the arguments of Tao, Vu, and Krishnapur in conjunction with \cref{thm:lsv-bound}, we show the following.

\begin{theorem}
\label{thm:universality-esd}
Let $M_{n} = (\mu_{i,j}^{(n)})_{i,j \leq n}$ and $C_{n} = (\sigma_{i,j}^{(n)})_{i,j\leq n}$ be constant (i.e. deterministic) matrices satisfying
\begin{enumerate}[(i)]
    \item $\sup_{n} n^{-2}\|M_{n}\|_{\HS}^{2} < \infty$;
    \item $\alpha \leq \sigma_{i,j}^{(n)}\leq \beta$ for all $n,i,j,$ for some $0 < \alpha < \beta < \infty$.
\end{enumerate}
Given a matrix $\boldsymbol{X} = (x_{i,j})_{i,j \leq n}$, set
$$A_{n}(\boldsymbol{X}) = M_{n} + C_{n}\cdot \boldsymbol{X} = (\mu_{i,j}^{(n)} + \sigma_{i,j}^{(n)}x_{i,j})_{i,j\leq n},$$
where "$\cdot$" denotes the Hadamard product. 
\begin{enumerate}
\item Suppose that $x_{i,j}^{(n)}$ are independent complex-valued random variables with $\E[x_{i,j}^{(n)}] = 0$ and $\E|x_{i,j}^{(n)}|^{2} = 1$, and that $y_{i,j}^{(n)}$ are independent complex-valued random variables, also having zero mean and unit variance. 

\item Assume that there exists some $b \in (0,1)$ such that $\Pr\left( |\widetilde{x_{i,j}}|\geq b\right)\geq b$ and similarly for $y_{i,j}$.

\item Also, assume Pastur's condition,
$$\frac{1}{n^{2}} \sum_{i,j=1}^{n} \E\left[|x_{i,j}^{(n)}|^{2}\boldsymbol{I}\{|x_{i,j}^{(n)}| \geq \epsilon \sqrt{n}\}|\right] \to 0 \quad \text{for all }\epsilon>0,$$
and the same for $\boldsymbol{Y}$ in place of $\boldsymbol{X}$. 
\end{enumerate}
Then,
$$\mu_{n^{-1/2}\cdot A_{n}(\boldsymbol{X})} - \mu_{n^{-1/2}\cdot A_{n}(\boldsymbol{Y})} \to 0$$
in the sense of probability. 
\end{theorem}

\begin{remark}
In \cite{tao2010random}, Krishnapur proved a similar result, except that the natural mild anti-concentration assumption 2.\ was replaced by the stronger, technical, and hard-to-verify condition that $x_{i,j}, y_{i,j}$ are dominated in a Fourier sense by a complex random variable with $\kappa$-controlled second moment (see \cite{tao2010random} for relevant definitions).
\end{remark}

\noindent {\bf Notation: } Throughout the paper, we will omit floors and ceilings when they make no essential difference. We will use $\S_{\C}^{n-1}$ to denote the set of unit vectors in $\C^{n}$, $B(x,r)$ to denote the ball of radius $r$ centered at $x$, and $\Re(\boldsymbol{v}), \Im(\boldsymbol{v})$ to denote the real and imaginary parts of a complex vector $\boldsymbol{v}\in \C^{n}$. As is standard, we will use $[n]$ to denote the discrete interval $\{1,\dots,n\}$. We will also use the asymptotic notation $\lesssim, \gtrsim, \ll, \gg$ to denote $O(\cdot), \Omega(\cdot), o(\cdot), \omega(\cdot)$ respectively. For a matrix $M$, we will use $\|M\|$ to denote its standard $\ell^{2}\to \ell^{2}$ operator norm and $\|M\|_{\HS}$ to denote the Hilbert-Schmidt norm. All logarithms are natural unless noted otherwise and $\star$ denotes the Schur (entry-wise) product.\\

\section{Anti-concentration for sums of non-identically distributed independent complex random variables}\label{sec:anti-concentration}

The goal of the theory of anti-concentration is to obtain upper bounds on the L\'evy concentration function, which is defined as follows.  

\begin{definition}[L\'evy concentration function]Let $X:=(X_1,\dots,X_n) \in \C^{n}$ be a complex random vector, and let ${v}:=(v_{1},\dots,v_{n})\in\C^{n}$. We define the \emph{L\'evy concentration function of ${v}$ at radius $r$ with respect to $X$} by $$\rho_{r,X}({v}):=\sup_{x\in\C}\Pr\left(v_{1}X_{1}+\dots+v_{n}X_{n}\in B(x,r)\right).$$
\end{definition}

Rudelson and Vershynin \cite{rudelson2008littlewood} introduced the notion of the essential least common denominator (LCD) to control the L\'evy concentration function. This notion was generalized in \cite{livshyts2019smallest} to the randomized least common denominator (RLCD) and used to handle non-i.i.d. real random variables. We give a generalization of this to non-i.i.d. complex random variables which will be useful for us.\\ 

\begin{definition}[CRLCD]
For a complex random vector ${X}:=(X_1,\dots,X_n) \in \C^{n}$, a deterministic vector ${{v}}:=(v_{1},\dots,v_{n})\in\C^{n}$, and parameters $L > 0, u\in (0,1)$, define 
$$\CRLCD_{L,u}^{X}({v}):=\inf_{\theta\in\C}\left\{|\theta|>0:\E[\dist^{2}(\theta {{v}}\star\widetilde{X},(\Z+i\Z)^{n})]<\min(u|\theta|^{2}\|{{v}}\|_{2}^{2},L^{2})\right\},$$
where $\widetilde{X}$ denotes the symmetrization of $X$ (i.e. $\widetilde{X}\sim X'-X''$, where $X'$ and $X''$ are independent copies of $X$). 
\end{definition}

Before proceeding to the results of this section, we need a couple of additional definitions. 

\begin{definition}[\cite{tao2008random}]
For a complex random vector ${X}:= (X_1,\dots,X_n) \in \C^{n}$ and a deterministic vector ${{v}}:= (v_1,\dots,v_n) \in \C^{n}$, let
$$P_{{X}}({{v}}):= \E\left[-\pi|\langle \widehat{X} , v\rangle|^{2}\right].$$
Here, $\widehat{X}:= \widetilde{X}\star (x_1,\dots,x_n)$, where $x_1,\dots,x_n$ are mutually independent $\Ber(1/2)$ random variables, which are also independent of $\widetilde{X}$.
\end{definition}

\begin{definition}[\cite{tao2008random}]
For a complex random variable $z \in \C$ and a fixed complex number $a \in \C$, let
$$\|a\|_{z}:= \left(\E\left[\|\Re(a\cdot \widetilde{z})\|_{\R / \Z}^{2}\right]\right)^{1/2},$$
where $\widetilde{z}$ denotes the symmetrization of $z$, and $\|\cdot \|_{\R / \Z}$ denotes the distance to the nearest integer. 
\end{definition}

\begin{lemma}[\cite{tao2008random}]
\label{lemma:tv-anticoncentration}
For a complex random vector $X:= (X_1,\dots,X_n)\in \C^{n}$ with independent coordinates, and deterministic vectors ${v}:= (v_1,\dots,v_n), {w}:= (w_1,\dots,w_n) \in \C^{n}$:
\begin{enumerate}
    \item $\rho_{r,X}(v) \leq \exp(\pi r^{2})\cdot P_X(v)$.
    \item $P_X(v)P_X(w) \leq 2P_{XX}(vw)$. Here, $vw \in \C^{2n}$ denotes the vector whose first $n$ coordinates coincide with $v$ and last $n$ coordinates coincide with $w$, and $XX \in \C^{2n}$ denotes the complex random vector whose first $n$ coordinates and last $n$ coordinates are both independent copies of $X$. 
    \item $P_X(v) \leq \int_{\C}\exp\left(-\sum_{i=1}^{n}\|\xi\cdot v_i \|^{2}_{X_i}/2\right)\exp(-\pi |\xi|^{2})d\xi$.
\end{enumerate}
\end{lemma}
\begin{proof}
1. follows from Lemma 4.3, 2. follows from Lemma 4.5 (iii), and 3. follows from Lemma 5.2 in \cite{tao2008random}. Actually, in \cite{tao2008random}, these results are stated only in the case when the coordinates of the random vector $X$ are identically distributed, but exactly the same proof also works for our more general setting.  
\end{proof}

Next, we need a small modification of the initial steps in the proof of Theorem 2.11 in \cite{jain2019strong}.
\begin{lemma}
\label{lemma:doubling}
Let $X := (X_1,\dots,X_n) \in \C^{n}$ be a complex random vector with independent coordinates, and let $w:= (w_1,\dots,w_n) \in \C^{n}$ be a deterministic vector. Then,
$$\rho_{r,X}(w)^{2} \leq 2\exp(2\pi r^{2})\cdot \int_{\C} \exp\left(-\frac{1}{2}\E\left[\dist^{2}\left(\xi w \star \widetilde{X}, (\Z+i\Z)^{n}\right)\right]\right)\exp(-\pi |\xi|^{2})d\xi.$$
\end{lemma}
\begin{proof}
Let $w_{\C} \in \C^{2n}$ denote the vector whose first $n$ coordinates coincide with $w$ and last $n$ coordinates coincide with $i\cdot w$. Then, since $\rho_{r,X}(w) = \rho_{r,X}(i\cdot w)$, we have
\begin{align*}
    \rho_{r,X}(w)^{2} 
    &= \rho_{r,X}(w)\rho_{r,X}(i\cdot w)\\
    &\leq \exp(2\pi r^{2})\cdot P_{X}(w)\cdot P_{X}(i \cdot w)\\
    &\leq 2\exp(2\pi r^{2})\cdot P_{XX}(w_{\C})\\
    &\leq 2\exp(2\pi r^{2})\cdot \int_{\C} \exp\left(-\frac{1}{2}\sum_{i=1}^{n}\left(\|\xi \cdot w_i\|_{X_i}^{2} + \|i\xi\cdot w_i\|_{X_i}^{2}\right)\right)\exp(-\pi |\xi|^{2})d\xi,
\end{align*}
  where the second, third and fourth inequalities follow from \cref{lemma:tv-anticoncentration} parts 1., 2., and 3. respectively. Finally, note that
  \begin{align*}
      \sum_{i=1}^{n}\left(\|\xi \cdot w_i\|_{X_i}^{2} + \|i\xi\cdot w_i\|_{X_i}^{2}\right)
      &= \E\sum_{i=1}^{n}\left(\|\Re(\xi w_i \cdot \widetilde{X_i}) \|^{2}_{\R / \Z} + \|\Re(i\xi w_i\cdot \widetilde{X_i}) \|^{2}_{\R / \Z}\right)\\
      &= \E\sum_{i=1}^{n}\left(\|\Re(\xi w_i \cdot \widetilde{X_i}) \|^{2}_{\R / \Z} + \|\Im(\xi w_i\cdot \widetilde{X_i}) \|^{2}_{\R / \Z}\right)\\
      &= \E\left[\dist^{2}(\xi w \star \widetilde{X}, (\Z + i\Z)^{n})\right]. 
      \qedhere
  \end{align*}
\end{proof}

The next proposition is the main result of this section. We note that the conclusion of this proposition can be considerably strengthened; however, the statement given below will be sufficient for our application, and has a much simpler proof.  
\begin{proposition}
\label{prop:anticonc-LCD}
Let $X:=(X_1,\dots,X_n)\in \C^{n}$ be a complex random vector with independent coordinates and let $v:=(v_1,\dots,v_n) \in \C^{n}$ be such that $\frac{1}{2}\leq \|v\|_{2} \leq 2$. Then, for all parameters $L>0, u\in (0,1)$, and for all $\epsilon > 0$,
$$\rho_{\epsilon,X}(v) \leq C_{\ref{prop:anticonc-LCD}}\left(\epsilon u^{-1/2} + \exp\left(-\frac{1}{4}L^{2}\right) + \exp\left(-\frac{\pi}{4}\epsilon^{2}\CRLCD_{L,u}^{X}(v)^{2}\right)\right),$$
where $C_{\ref{prop:anticonc-LCD}}$ is an absolute constant.   
\end{proposition}
\begin{proof}
  Let ${w}:={v}/\epsilon\in\C^{n}$. Then, $2^{-1}\epsilon^{-1}\leq\|w\|_{2}\leq2\epsilon^{-1}$
and $\rho_{\epsilon, X}(v) = \rho_{1,X}(w)$.  
Moreover,
\begin{align*}
    \rho_{1,X}(w)^{2} 
    &\leq 2\exp(2\pi)\cdot \int_{\C}\exp\left(-\frac{1}{2}\E\left[\dist^{2}\left(\xi w\star \widetilde{X}, (\Z + i\Z)^{n}\right)\right]\right)\exp(-\pi|\xi|^{2})d\xi\\
    &= 2\exp(2\pi)\epsilon^{2}\cdot \int_{\C}\exp\left(-\frac{1}{2}\E\left[\dist^{2}\left(\eta v\star \widetilde{X}, (\Z + i\Z)^{n}\right)\right]\right)\exp(-\pi\epsilon^{2}|\eta|^{2})d\eta
\end{align*}
where the first line follows from \cref{lemma:doubling} and the second line follows from the change of variables $\xi = \epsilon \eta$. 

Let $$F(\eta) = \exp\left(-\frac{1}{2}\E\left[\dist^{2}\left(\eta v\star \widetilde{X}, (\Z + i\Z)^{n}\right)\right]\right)\exp(-\pi\epsilon^{2}|\eta|^{2}).$$ We break the above integral into two regions, $B(0,\CRLCD^{X}_{L,u}(v))$ and $\C \setminus B(0,\CRLCD^{X}_{L,u}(v))$.


For the first region, note that by the definition of $\CRLCD$, 
\begin{align*}
\int_{B(0,\CRLCD_{L,u}^{X}(v))} F(\eta) d\eta & \leq\int_{B(0,\CRLCD_{L,u}^{X}(v))}\exp\left(-\frac{1}{2}\min\left(u|\eta|^{2}\|v\|_{2}^{2},L^{2}\right)-\pi\epsilon^{2}|\eta|^{2}\right)d\eta\\
& \leq\int_{\C}\exp\left(-\frac{1}{2}\min\left(u|\eta|^{2}\|v\|_{2}^{2},L^{2}\right)-\pi\epsilon^{2}|\eta|^{2}\right)d\eta\\
& \leq \int_{\C}\exp\left(-\frac{1}{2}u|\eta|^{2}\|v\|_{2}^{2} \right)d\eta + \int_{\C}\exp\left(-\frac{1}{2}L^{2} - \pi \epsilon^{2}|\eta|^{2}\right)d\eta \\
&\leq C\left(u^{-1} + \epsilon^{-2}\cdot \exp\left(-\frac{1}{2}L^{2}\right)\right),
\end{align*}
for some absolute constant $C > 0$. For the second region, note that 
\begin{align*}
\int_{\C\setminus B(0,\CRLCD_{L,u}^{X}(v))} F(\eta) d\eta & \leq\int_{\C\setminus B(0,\CRLCD_{L,u}^{X}(v))}\exp(-\pi\epsilon^{2}|\eta|^{2})d\eta\\
 & =\epsilon^{-2}\int_{\C\setminus B(0,\epsilon \CRLCD_{L,u}^{X}(v))}\exp(-\pi|\xi|^{2})d\xi\\
 & \leq C\epsilon^{-2}\exp\left(-\frac{\pi}{2}\epsilon^{2}\CRLCD_{L,u}^{X}(v)^{2}\right),
\end{align*}
for some absolute constant $C > 0$.
Putting everything together, we see that 
\begin{align*}
    \rho_{\epsilon,X}(v)^{2}\leq C\left(\epsilon^{2}u^{-1} + \exp\left(-\frac{1}{2}L^{2}\right) + \exp\left(-\frac{\pi}{2}\epsilon^{2}\CRLCD^{X}_{L,u}(v)^{2}\right) \right),
\end{align*}
so that
$$\rho_{\epsilon,X}(v) \leq C\left(\epsilon u^{-1/2} + \exp\left(-\frac{1}{4}L^{2}\right) + \exp\left(-\frac{\pi}{4}\epsilon^{2}\CRLCD_{L,u}^{X}(v)^{2}\right)\right),$$
as desired. 
\end{proof}

We conclude this section with the following lemma, which shows that weighted sums of random variables with finite non-zero variance are not too close to being a constant. 
\begin{lemma}
\label{lemma:anticoncentration}
Let $X:=(X_1,\dots,X_n) \in \C^{n}$ be a complex random vector with independent coordinates such that $\Pr\left(b^{-1}\geq |\widetilde{X}_i|\geq b\right) \geq b$ for some $b\in (0,1)$. There exists a constant $c_{\ref{lemma:anticoncentration}} \in (0,1)$ depending only on $b$ such that for all unit vectors $v := (v_1,\dots,v_n) \in \S_{\C}^{n-1}$, 
\begin{equation}\label{eq:anticoncentration}
\rho_{c_{\ref{lemma:anticoncentration}},X}(v) \leq 1-c_{\ref{lemma:anticoncentration}}.
\end{equation}

\end{lemma}

\begin{proof}
Let $M$ be a sufficiently large constant depending only on $b$, to be determined during the course of the proof. We consider two cases, depending on $\|v\|_{\infty}$.\\

\textbf{Case I: }$\|v\|_{\infty} \geq M^{-1}$. Without loss of generality, suppose $|v_{1}| > M^{-1}$. Then, by conditioning on the variables $X_2,\dots,X_n$, we see that it suffices to prove that $\rho_{c,X_1}(v_1) \leq 1-c$, for some constant $c\in (0,1)$ depending only on $b$ (and $M$). But this follows immediately since $\Pr(|\widetilde{X_1}|\geq b)\geq b$.\\

\textbf{Case II: }. $\|v\|_{\infty} < M^{-1}$. In this case, it suffices to show that $\CRLCD_{L,u}^{X}(v) \geq Mb$, for $u = b^{3}$ and all $L$ sufficiently large, for then, \cref{eq:anticoncentration} follows immediately from \cref{prop:anticonc-LCD} by taking $M$ to be sufficiently large depending on $b$. 


In order to show this, by definition, it suffices to show that for all $\theta \in \C$ such that $0< |\theta| < Mb$, 
$$\E\left[\dist^{2}(\theta v \star \widetilde{X}, (\Z+i\Z)^{n}\right] \geq u|\theta|^{2}\|v\|_{2}^{2}.$$
For this, we begin by noting that for any such value of $\theta$,
\begin{align*}
\dist^{2}(\theta v \star \widetilde{X}, (\Z+i\Z)^{n}) 
&\geq \sum_{i=1}^{n}|\theta|^{2}|v_i|^2|\widetilde{X}_i|^{2}\boldsymbol{1}\left[|\theta v_i\widetilde{X}_i| \leq \frac{1}{10} \right]\\
&\geq \sum_{i=1}^{n}|\theta|^{2}|v_i|^2|\widetilde{X}_i|^{2}\boldsymbol{1}\left[|\widetilde{X}_i| \leq b^{-1} \right]\\
&\geq \sum_{i=1}^{n}|\theta|^{2}|v_i|^2|\widetilde{X}_i|^{2}\boldsymbol{1}\left[b\leq |\widetilde{X}_i| \leq b^{-1} \right]\\
&\geq \sum_{i=1}^{n}b^{2}|\theta|^{2}|v_i|^2\boldsymbol{1}\left[b\leq |\widetilde{X}_i| \leq b^{-1}  \right].
\end{align*}
Therefore, taking the expectation on both sides, we see that
\begin{align*}
    \E\left[\dist^{2}(\theta v \star \widetilde{X}, (\Z+i\Z)^{n})\right]
    &\geq \sum_{i=1}^{n}b^{2}|\theta|^{2}|v_i|^2\E\left[\boldsymbol{1}\left[b\leq |\widetilde{X}_i| \leq b^{-1}  \right]\right]\\
    &\geq \sum_{i=1}^{n}b^{2}|\theta|^{2}|v_i|^2\cdot b\\
    &= b^{3}|\theta|^{2}\|v\|_{2}^{2}, 
\end{align*}
which gives the desired conclusion. 
\end{proof}

\section{Proof of Theorem \ref{thm:lsv-bound}}
In this section, we prove \cref{thm:lsv-bound} following \cite{livshyts2018smallest, livshyts2019smallest}. The only new ingredients are \cref{lemma:invertibility-single-vector} and \cref{lemma:incomp-anti-concentration}.\\

The first step in the proof of Theorem \ref{thm:lsv-bound} is to give a decomposition of the sphere $\S_{\C}^{n-1}$. For some parameters $\delta, \rho \in (0,1)$ to be chosen later, we define the sets of sparse, compressible, and incompressible vectors as follows:
\begin{align*}
    \Sparse(\delta)&:= \left\{ u\in \S_{\C}^{n-1}: \supp(u) \leq \delta n\right\},\\
    \Comp(\delta,\rho)&:=\left\{u \in \S_{\C}^{n-1}: \dist(u,\Sparse(\delta))\leq \rho \right\},\\
    \Incomp(\delta,\rho)&:= \S_{\C}^{n-1}\setminus \Comp(\delta,\rho).
\end{align*}
This results in
$$\S_{\C}^{n-1} = \Comp(\delta,\rho) \cup  \Incomp(\delta,\rho).$$ By characterization \eqref{eq:characterization}, we have 
$$ \Pr(s_n(A_n)\leq \epsilon\cdot n^{-1/2}) \le  \Pr\left( \inf_{x \in \Comp(\delta,\rho) }  \|A_nx\|_2 \leq \epsilon\cdot n^{-1/2}\right) + \Pr\left( \inf_{x \in \Incomp(\delta,\rho) } \|A_nx\|_2  \leq \epsilon\cdot n^{-1/2} \right).$$
We first deal with the compressible vectors. For this, as is standard, we begin with an estimate for `invertibility with respect to a single vector', which in our case, follows directly by combining  \cref{lemma:anticoncentration} with the so-called tensorization lemma (see Lemma 2.2 in \cite{rudelson2008littlewood}). 
\begin{lemma}
\label{lemma:invertibility-single-vector}
Let $A_{N,n}$ be an $N\times n$ complex random matrix whose entries $A_{i,j}$ are independent and satisfy  $\Pr\left(b^{-1} \geq |\widetilde{A_{i,j}}| \geq b \right) \geq b$ for some $b\in (0,1)$ .Then, for any fixed $v\in \S_{\C}^{n-1}$,
$$\Pr\left(\|A_{N,n}v\|_2\leq c_{\ref{lemma:invertibility-single-vector}}\sqrt{N}\right) \leq (1- c_{\ref{lemma:invertibility-single-vector}})^{N},$$
where $c_{\ref{lemma:invertibility-single-vector}}\in (0,1)$ is a constant depending only on $b$. 
\end{lemma}

The following crucial theorem about the existence of a suitable net on the sphere parameterized by $\|A\|_{\HS}$ follows from Corollary $4$ of \cite{livshyts2018smallest}. Note that our net is a subset of $\C^{n}$ rather than $\R^n$ as proven in \cite{livshyts2018smallest} but the same argument used there also works in our setting with slightly worse constants.

\begin{theorem}\label{thm:net}
Fix $N,n \in \N$ and consider any subset $S \subset \S_{\C}^{n-1}$. For any $\mu \in (0,1)$, and for every $\epsilon \in (0,\mu^{c_0})$ (for some absolute constant $c_0 > 0$), there exists a deterministic net $\mathcal{N} \subset \C^{n}$, with
$$|\mathcal{N}| \leq N(S,\epsilon B_{2}^{n})\cdot (O(\epsilon))^{\mu n},$$
and there exist positive constants $C_1(\mu), C_2(\mu)$ such that for every random $N\times n$ complex random matrix $A_{N,n}$ with independent columns, with probability at least
$$1 - e^{-C_1(\mu)n},$$
for every $x\in S$, there exists $y \in \mathcal{N}$ so that
$$\|A_{N,n}(x-y)\|_{2} \leq \frac{C_2(\mu)\epsilon}{\sqrt{n}}\sqrt{\E[\|A\|_{\HS}^{2}]}.$$
Here, $N(S, \epsilon B_{2}^{n})$ denotes the covering number of $S$ by $\epsilon B_{2}^{n}$. 
\end{theorem}

Using \cref{thm:net} and the invertibility with respect to a single vector from  \cref{lemma:invertibility-single-vector}, the following anti-concentration result for compressible vectors follows identically from Lemma $5.3$ of \cite{livshyts2018smallest}.
 
\begin{proposition}\label{prof:compressible-bound}
Let $A$ be an $ n\times n$ random matrix whose entries $A_{i,j}$ are independent and satisfy $\E\|A\|_{\HS}^{2} \leq Kn^{2}$ for some $K > 0$, and $\Pr\left(|\widetilde{A_{i,j}}|\geq b\right) \geq b$ for some $b\in (0,1)$. Then,
$$\Pr\left(\inf_{x\in \Comp(\delta,\rho)}\|Ax\|_{2} \leq C_{\ref{prof:compressible-bound}}\sqrt{n} \right) \leq 2e^{-c_{\ref{prof:compressible-bound}n}},$$
where $\rho,\delta \in (0,1)$ and $C_{\ref{prof:compressible-bound}},c_{\ref{prof:compressible-bound}} > 0$ depend only on $K$ and $b$. 
\end{proposition}

For the incompressible vectors, we use an `invertibility via distance' bound similar to \cite{rudelson2008littlewood}. The precise version we use appears in \cite{livshyts2019smallest}.

\begin{lemma}[Invertibility via distance, Lemma 6.1 in \cite{livshyts2019smallest}]
\label{lemma:invertibility-via-distance}
Fix a pair of parameters $\delta, \rho \in (0,1/2)$, and assume that $n \geq 4/\delta$. Then, for any $\epsilon > 0$,
$$\Pr\left(\inf_{x\in \Incomp(\delta,\rho)}\|Ax\|_2 \leq \epsilon \frac{\rho}{\sqrt{n}}\right) \leq \frac{4}{\delta n} \inf_{I \subset [n], |I| = n - \lfloor \delta n/2 \rfloor}\sum_{j \in I}\Pr\left(\dist(A_j, H_j) \leq \epsilon\right),$$
where $H_j$ denotes the subspace spanned by all the columns of $A$ except for $A_j$. 
\end{lemma}
From the previous lemma, to control $\|Ax\|_2$ for $x$ over the incompressible vectors, it suffices to understand the anti-concentration of $\dist(A_j, H_j)$ where $H_j$ denotes the subspace spanned by all the columns of $A$ except $A_j$. For this, we begin by noting that $\dist(A_j, H_j) \geq |\langle A_j, \nu_{j} \rangle|$, where $\nu_{j}$ denotes any unit vector normal to $H_j$, so that anti-concentration of $\dist(A_j, H_j)$ reduces to studying the anti-concentration properties of a unit normal to a random hyperplane. Before proceeding to the details, we will need the following lemma, which shows that incompressible vectors have sufficiently large CRLCD.  

\begin{lemma}[Incompressible vectors have large CRLCD, Lemma $2.10$ in \cite{livshyts2019smallest}]
\label{lemma:incomp-large-CRLCD}
For any $b,\delta, \rho \in (0,1)$, there are $n_0 = n_0(b,\delta, \rho)$, $h_{\ref{lemma:incomp-large-CRLCD}} = h_{\ref{lemma:incomp-large-CRLCD}}(b,\delta,\rho)\in (0,1)$ and $u_{\ref{lemma:incomp-large-CRLCD}} = u_{\ref{lemma:incomp-large-CRLCD}}(b,\delta,\rho) \in (0,1/4)$ with the following property. Let $n \geq n_0$, let $v \in \Incomp_{n}(\delta,\rho)$, and assume that a random vector $X = (X_1,\dots,X_n)$ with independent components satisfies $ \Pr\left(|\widetilde{X}_{i}| \geq b\right) \geq b$ for all $1\leq i\leq n$, and $\E\|X\|^{2} \leq T$, for some fixed parameter $T  \gtrsim n$. Then, for any $L > 0$, we have
$$\CRLCD_{L,u_{\ref{lemma:incomp-large-CRLCD}}}^{X}(v) \geq h_{\ref{lemma:incomp-large-CRLCD}}\cdot \frac{n}{\sqrt{T}} $$
\end{lemma}
\begin{remark}
In \cite{livshyts2019smallest}, the above proposition is proved for the notion of RLCD defined there, but the same proof goes through for the CRLCD as well. 
\end{remark}

We can now prove the desired invertibility on incompressible vectors.

\begin{proposition} \label{lemma:incomp-anti-concentration}
Let $A$ be an $ n\times n$ random matrix whose entries $A_{i,j}$ are independent and satisfy $\E\|A\|_{\HS}^{2} \leq Kn^{2}$ for some $K > 0$, and $\Pr\left(b^{-1} \geq |\widetilde{A_{i,j}}|\geq b\right) \geq b$ for some $b\in (0,1)$.
Fix a pair of parameters $\delta, \rho \in (0,1/2)$, and assume that $n \ge 4/\delta$. There exists absolute constants $C_{\ref{lemma:incomp-anti-concentration}}, c'_{\ref{lemma:incomp-anti-concentration}}$ that only depend on $\delta, \rho, b, K$ such that for any $\epsilon \in (0,1)$, 
$$\Pr\left(\inf_{x\in \Incomp(\delta,\rho)}\|Ax\|_2 \leq \epsilon \frac{\rho}{\sqrt{n}}\right) \leq C_{\ref{lemma:incomp-anti-concentration}} \left( \epsilon  +  \exp(-c'_{\ref{lemma:incomp-anti-concentration}}\epsilon^{2}n)\right).$$
\end{proposition}

\begin{proof}

Fix $\delta \in (0, 1/2)$. Since $\E\|A\|_{\HS}^2 \le K n^2$, there must be at least $(1-\delta/4)n$ columns $A_i$ that satisfy $\E \|A_i\|_2^2 \le 4Kn/\delta$. Let $I$ denote an arbitrary set of such indices of size $n - \lfloor \delta n/2 \rfloor$. We will apply \cref{lemma:invertibility-via-distance} with this choice of $I$. 

For this, fix $i \in I$, and let $H_i$ denote the span of all columns of the matrix except for $A_i$. Then, an identical argument to \cref{prof:compressible-bound} shows that, except with probability at most $\exp(-c_{\ref{prof:compressible-bound}}n/2)$, any unit vector $\nu$ which is orthogonal to $H_i$ must belong to $\Incomp(\delta',\rho')$, where $\delta', \rho'$ depend only on $K$ and $b$. Henceforth, we restrict ourselves to this event, and let $\nu$ denote a unit normal vector to the (random) hyperplane $H_i$. 

By \cref{lemma:incomp-large-CRLCD}, it follows that for any $L > 0$, 
$$\CRLCD^{A_j}_{L, u_{\ref{lemma:incomp-large-CRLCD}}}(\nu) \geq C(b,K)\sqrt{n}.$$
Therefore, by \cref{prop:anticonc-LCD}, it follows that
$$ \Pr\left(\dist(A_j, H_j) \leq \epsilon\right) \le  \rho_{2\epsilon,A_j}(\nu) \le C_{\ref{prop:anticonc-LCD}}\left(2\epsilon u_{\ref{lemma:incomp-large-CRLCD}}^{-1/2} + \exp\left(-\frac{1}{4}L^{2}\right) + \exp\left(-C'(b,K)\epsilon^{2}n \right)\right).$$
Finally, taking $L > 2\sqrt{C'(b,K) n}$, and using \cref{lemma:invertibility-via-distance} gives the desired conclusion. 
\end{proof}

\begin{proof}[Proof of Theorem \ref{thm:lsv-bound}]
The proof of \cref{thm:lsv-bound} follows from using characterization \eqref{eq:characterization} and combining Proposition \cref{prof:compressible-bound} and \cref{lemma:incomp-anti-concentration}. 
\end{proof}

\section{Proof of \cref{thm:universality-esd}}
By means of the so-called \emph{replacement principle} (Theorem 2.1 in \cite{tao2010random}), the following analogue of Proposition 2.2 in \cite{tao2010random} suffices to prove \cref{thm:universality-esd}. 

\begin{proposition}
\label{prop:converging-det}
Let $A_{n}(\boldsymbol{X})$ and $A_{n}(\boldsymbol{Y})$ be as in the statement of \cref{thm:universality-esd}. Then, for every fixed $z \in \C$, 
$$\frac{1}{n}\log\left|\det\left(\frac{1}{\sqrt{n}}A_{n}(\boldsymbol{X}) - zI\right)\right| - \frac{1}{n}\log\left|\det\left(\frac{1}{\sqrt{n}}A_{n}(\boldsymbol{Y}) - zI\right)\right|$$
converges in probability to zero. 
\end{proposition}

By using Steps 2,3,4 in the proof of Theorem C.2 in \cite{tao2010random} verbatim, the proof of \cref{prop:converging-det} is reduced to proving the following.

\begin{proposition}
\label{prop:log-det-sum}
Let $A_{n}(\boldsymbol{X})$ and $A_{n}(\boldsymbol{Y})$ be as in the statement of \cref{thm:universality-esd}, and $z \in \C$ be fixed. Let $X_1,\dots,X_n$ be the rows of $A_{n}(\boldsymbol{X}) - \sqrt{n} zI$ and, for each $1\leq i\leq n$, let $V_i$ be the $(i-1)$-dimensional space generated by $X_1,\dots,X_{i-1}$. Similarly, let $Y_1,\dots,Y_{n}$ be the rows of $A_{n}(\boldsymbol{Y}) - \sqrt{n}zI$ and, for each $1\leq i \leq n$, let $W_i$ be the $(i-1)$-dimensional space generated by $Y_1,\dots,Y_{i-1}$. Then, 
$$\frac{1}{n} \sum_{n-n^{0.99} \leq i \leq n}\left(\log\dist\left(\frac{1}{\sqrt{n}}X_i, V_i\right) - \log\dist\left(\frac{1}{\sqrt{n}}Y_i, W_i\right)\right) $$
converges in probability to zero. 
\end{proposition}

We can further reduce to proving the following high probability bounds on the extreme singular values of $A_{n}(\boldsymbol{X})$ and $A_{n}(\boldsymbol{Y})$.

\begin{proposition}
\label{prop:control-lsv-enough}
Let $A_{n}(\boldsymbol{X})$ and $A_{n}(\boldsymbol{Y})$ be as in the statement of \cref{thm:universality-esd}, and $z \in \C$ be fixed. Then, there exists an absolute constant $C > 0$ such that 

\begin{enumerate}
    \item $\Pr\left(\sigma_{1} (A_{n}(\boldsymbol{X})-z\sqrt{n}I) \geq  n^C \right) = o_n(1),$
    \item $\Pr\left(\sigma_{n} (A_{n}(\boldsymbol{X}) - z\sqrt{n}I) \leq  n^{-C} \right) = o_n(1),$
\end{enumerate}
and similarly for $A_n(\boldsymbol{Y})$.
\end{proposition}
Before proving \cref{prop:control-lsv-enough}, let us show how it implies \cref{prop:log-det-sum}. We will make use of the following linear algebraic fact. 

\begin{lemma}[Lemma A.4 in \cite{tao2010random}]
\label{lemma:negative-second-moment}
Let $A$ be an invertible $n\times n$ matrix with singular values $\sigma_{1}(A) \geq \dots \geq \sigma_{n}(A) > 0$ and rows $X_1,\dots, X_n \in \C^{n}$. For each $1\leq i \leq n$, let $U_i$ be the hyperplane generated by the $n-1$ rows $X_1,\dots,X_{i-1},X_{i+1},\dots,X_{n}$. Then,
$$\sum_{j=1}^{n}\sigma_{j}(A)^{-2} = \sum_{j=1}^{n}\dist(X_j,U_j)^{-2}.$$
\end{lemma}

\begin{proof}[\cref{prop:control-lsv-enough} implies \cref{prop:log-det-sum}]
For $1\leq i \leq n$, let $U_i$ denote the hyperplane generated by the $n-1$ rows $X_1,\dots,X_{i-1}, X_{i+1},\dots,X_{n}$ of $A_{n}(\boldsymbol{X}) - \sqrt{n}zI$. First, note that
$$\frac{1}{\sqrt{n}}\dist(X_i, U_i) = \dist\left(\frac{1}{\sqrt{n}}X_i, U_i\right) \leq \dist\left(\frac{1}{\sqrt{n}}X_i, V_i\right) \leq \frac{1}{\sqrt{n}}\dist(X_i, 0) = \frac{1}{\sqrt{n}}\|X_i\|,$$
and similarly for $A_{n}(\boldsymbol{Y}) - \sqrt{n} z I$. 
Next, by  \cref{lemma:negative-second-moment},
$$\dist\left(\frac{1}{\sqrt{n}} X_i, V_i\right) \geq \dist\left(\frac{1}{\sqrt{n}} X_i, U_i\right) 
\geq \frac{1}{n}\sigma_{n}(A_{n}(\boldsymbol{X})-\sqrt{n}zI),$$
and similarly for $A_{n}(\boldsymbol{Y}) - \sqrt{n}zI$. 

Therefore, \cref{prop:log-det-sum} follows if we can show that, except with probability $o_{n}(1)$, $\|X_i\| \leq n^{O(1)}$ (for all $1\leq i \leq n$), $\sigma_{n}(A_{n}(\boldsymbol{X})- \sqrt{n}zI) \geq n^{-O(1)}$, and similarly for $A_{n}(\boldsymbol{Y}) - \sqrt{n}zI$. Indeed, in this case, except with probability $o_n(1)$, each summand of the sum appearing in \cref{prop:log-det-sum} is bounded in absolute value by $O(\log{n})$, so that the entire sum is bounded in absolute value by
$$ \frac{1}{n}\cdot O(\log{n})\cdot n^{0.99} \leq O\left(\frac{1}{n^{0.001}}\right).$$

Finally, note that for all $1\leq i\leq n$, $\|X_i\| \leq \sigma_{1}(A_n(\boldsymbol{X})- \sqrt{n} zI)$, so that the desired probability bounds on $\|X_i\|$ and $\sigma_{n}(A_n(\boldsymbol{X}))$ (and similarly for $A_n(\boldsymbol{Y})$) follow from \cref{prop:control-lsv-enough}. 
\end{proof}

\begin{proof}[Proof of \cref{prop:control-lsv-enough}]
\textbf{Bound on $\sigma_{1}$: }By the triangle inequality for $\sigma_{1} (= \|\cdot \|)$, it suffices to show that there is an absolute constant $C > 0$ such that $\Pr\left(\sigma_{1}(A_{n}(\boldsymbol{X})) \geq n^{C}\right) = o_{n}(1)$. Note that by assumptions (i), (ii) and 1. in the statement of \cref{thm:universality-esd}, we have $\E[\|A_{n}(\boldsymbol{X})\|_{\HS}^{2}] = O(n^{C'})$, so that $\E[\sigma_{1}^{2}(A_{n}(\boldsymbol{X}))] = O(n^{C'})$. The desired conclusion now follows from Markov's inequality.\\ 

\textbf{Bound on $\sigma_{n}$: }
We begin by verifying that $P:= A_{n}(\boldsymbol{X}) - \sqrt{n}zI$ satisfies the assumptions of \cref{thm:lsv-bound}. An identical argument works for $A_{n}(\boldsymbol{Y}) - \sqrt{n} z I$ as well. 

Assumptions (i), (ii), and 1.\ of \cref{thm:universality-esd} show that $\E \sum_{i,j}|P_{i,j}|^{2} \leq Kn^{2}$ for some $K > 0$.

Moreover, assumptions 1.\ and 2.\ of \cref{thm:universality-esd} show that there exists some $b' \in (0,1)$ such that $\Pr(b'^{-1}|\widetilde{P_{i,j}}| \geq b') \geq b'$ for all $i,j$ -- indeed, assumption 2.\ of \cref{thm:universality-esd} shows that $\Pr(|\widetilde{P_{i,j}}|\geq b/\beta)\geq b$, whereas assumption 1.\ shows that $\E[|\widetilde{P_{i,j}}|^{2}] \leq \beta^{2}$, so that by Markov's inequality, $\Pr(|\widetilde{P_{i,j}}|\geq \beta \cdot \sqrt{2/b} ) \leq b/2$, so that we can conclude using the union bound and taking $b'$ to be sufficiently small.

Finally, we can apply \cref{thm:lsv-bound} to $P$ with $\epsilon = n^{-1/4}$ (say) to obtain the desired conclusion. 
\end{proof}

\bibliographystyle{abbrv}
\bibliography{least-singular-value}

\begin{thebibliography}{10}

\bibitem{edelman1988eigenvalues}
A.~Edelman.
\newblock Eigenvalues and condition numbers of random matrices.
\newblock {\em SIAM Journal on Matrix Analysis and Applications},
  9(4):543--560, 1988.

\bibitem{jain2019strong}
V.~Jain.
\newblock The strong circular law: a combinatorial view.
\newblock {\em arXiv preprint arXiv1904.11108}, 2019.

\bibitem{livshyts2018smallest}
G.~V. Livshyts.
\newblock The smallest singular value of heavy-tailed not necessarily iid
  random matrices via random rounding.
\newblock {\em arXiv preprint arXiv:1811.07038}, 2018.

\bibitem{livshyts2019smallest}
G.~V. Livshyts, K.~Tikhomirov, and R.~Vershynin.
\newblock The smallest singular value of inhomogeneous square random matrices.
\newblock {\em arXiv preprint arXiv:1909.04219}, 2019.

\bibitem{rebrova2018coverings}
E.~Rebrova and K.~Tikhomirov.
\newblock Coverings of random ellipsoids, and invertibility of matrices with
  iid heavy-tailed entries.
\newblock {\em Israel Journal of Mathematics}, 227(2):507--544, 2018.

\bibitem{rudelson2008invertibility}
M.~Rudelson.
\newblock Invertibility of random matrices: norm of the inverse.
\newblock {\em Annals of Mathematics}, pages 575--600, 2008.

\bibitem{rudelson2008littlewood}
M.~Rudelson and R.~Vershynin.
\newblock The {L}ittlewood--{O}fford problem and invertibility of random
  matrices.
\newblock {\em Advances in Mathematics}, 218(2):600--633, 2008.

\bibitem{tao2008random}
T.~Tao and V.~Vu.
\newblock Random matrices: the circular law.
\newblock {\em Communications in Contemporary Mathematics}, 10(02):261--307,
  2008.

\bibitem{tao2010smooth}
T.~Tao and V.~Vu.
\newblock Smooth analysis of the condition number and the least singular value.
\newblock {\em Mathematics of computation}, 79(272):2333--2352, 2010.

\bibitem{tao2010random}
T.~Tao, V.~Vu, and M.~Krishnapur.
\newblock Random matrices: Universality of {ESD}s and the circular law.
\newblock {\em The Annals of Probability}, 38(5):2023--2065, 2010.

\end{thebibliography}

\appendix

\end{document}